\documentclass[11pt]{article}
\usepackage{latexsym}
\usepackage{comment}
\usepackage{amsmath,amsthm,amsfonts,amssymb,graphicx,epsfig,latexsym,color,mathrsfs}
\usepackage{epsf,enumerate}

\newtheorem{thm}{Theorem}[section]
\newtheorem{lemma}[thm]{Lemma}
\newtheorem{cor}[thm]{Corollary}
\newtheorem{prop}[thm]{Proposition}

{}

\newcommand*{\mysim}{\mathord{\sim}}
\theoremstyle{remark}

\newtheorem{remark}[thm]{Remark}

\newtheorem*{definition*}{Definition}
\newtheorem*{remark*}{Remark}

\newtheorem{claim}[thm]{Claim}

\def\R{{\mathbb R}}

\def\Z{{\mathbb Z}}

\def\F{{\mathcal F}}

\renewcommand{\>}{\rangle}

\renewcommand{\phi}{\varphi}
\renewcommand{\bar}{\overline}
\newcommand{\cvn}{\text{cv}_N}
\newcommand{\CVN}{\text{CV}_N}
\newcommand{\barcvn}{\bar{\text{cv}}_N}
\newcommand{\barCVN}{\bar{\text{CV}}_N}
\newcommand{\FN}{F_N}
\newcommand{\Stab}{\text{Stab}}

\newcommand{\ssm}{\smallsetminus}

\title{The boundary of the complex of free factors}
\author{Mladen Bestvina and Patrick Reynolds\thanks{The
    first author gratefully acknowledges the support by the National
    Science Foundation.}}
\date{March 22, 2013}
\begin{document}

\maketitle

\begin{abstract}
 We give a description of the boundary of a complex of free factors
 that is analogous to E. Klarreich's description of the boundary of a
 curve complex. The argument uses the geometry of folding paths
 developed in \cite{BF11} and the structure theory of trees on the
 boundary of Outer space developed recently by Coulbois, Hilion, Lustig and
 Reynolds.
\end{abstract}

\section{Introduction}

The complex of free factors, denoted $\mathcal{F}=\mathcal{F}_N$, for
the free group $\FN$ is an analogue for the complex of curves for a
surface.  The simplicial complex $\mathcal{F}$ arises as the nerve of
the intersection pattern for thin regions in Outer space, and hence
codes the geometry of Outer spaces relative to these thin regions.
Vertices of $\mathcal{F}$ are conjugacy classes of non-trivial proper
free factors of the rank-$N$ free group $\FN$, and higher dimensional
simplices correspond to chains of inclusions of free factors.

Equip $\mathcal{F}$ with the simplicial metric.  It was shown in
\cite{BF11} that $\mathcal{F}$ is Gromov hyperbolic; the goal of the
present note is to give a concrete description of the boundary
$\partial \mathcal{F}$ of $\mathcal{F}$. Kapovich-Rafi \cite{kr12}
have shown that hyperbolicity of $\mathcal{F}$ can be deduced from the
hyperbolicity of the free splitting complex, which was shown by
Handel-Mosher \cite{HM12}, and an alternative proof of this was given
by Hilion-Horbez \cite{hh12}.

Let $\partial \CVN$ denote the boundary of the Culler-Vogtmann Outer
space $\CVN$ \cite{CV86}; the points of $\partial \CVN$ are
represented by very small actions of $\FN$ on $\mathbb{R}$-trees.
Associated to $T \in \partial \CVN$ is a (algebraic) lamination
$L(T)$, which intuitively records information about which elements of
$\FN$ act with short translation length in $T$.  A lamination is an
$\FN$-invariant, flip invariant, closed subset $X \subseteq \partial^2
\FN =\partial \FN \times \partial \FN \ssm diag$.  A finitely
generated subgroup $H \leq \FN$ is a virtual retract of $\FN$ by
M. Hall's Theorem, hence $H$ is quasi-convex in $\FN$, and $\partial
H$ embeds in $\partial \FN$; say that $H$ carries a leaf of $X$ if $X
\cap \partial^2 H \neq \emptyset$.

A lamination $X$ is called arational if no leaf of $X$ is carried by a
proper free factor of $\FN$; a tree $T \in \partial \CVN$ is called
arational if $L(T)$ is arational.  Let $\mathcal{AT} \subseteq
\partial \CVN$ denote the set of arational trees, equipped with the
subspace topology.  Define a relation $\sim$ on $\mathcal{AT}$ by $S
\sim T$ if and only if $L(S)=L(T)$, and give $\mathcal{AT}/\mysim$ the
quotient topology.  Our main result is:

\begin{thm}
 The space $\partial \mathcal{F}$ is homeomorphic to $\mathcal{AT}/\mysim$.
\end{thm}

This theorem is a very strong analogue of E. Klarreich's description
of the boundary $\partial \mathcal{C}(S)$ of the complex of curves
$\mathcal{C}(S)$ associated to a non-exceptional surface $S$;
Klarreich showed that $\partial \mathcal{C}(S)$ is homeomorphic to
$\mathcal{AF}/\mysim$, where $\mathcal{AF} \subseteq \mathcal{PML}(S)$
is the subspace consisting of arational measured foliations, and where
for $E,F \in \mathcal{AF}$ one has $E \sim F$ if and only if the
underlying topological foliations are equivalent \cite{Kla}.

The argument for our main result follows the outline of Klarreich's
paper, but the details are quite different; the difficulty comes from
pushing an analogy between Outer space and Teichm\"uller space.

The paper is organized as follows.  Relevant
background about Outer space, very small $\FN$-trees, laminations, and
$\mathcal{F}$ is found in Section 2.  The proof of the main result can
be roughly divided into four steps.  The first step is to show that
arational trees are indeed very close analogues of arational measured
foliations on surfaces; this is accomplished in Sections 3 and 4.  The
main result is Theorem \ref{T.UniqueLam}, a duality result, which says
that if $T$ is arational and shares a length 0 current with $S$, then
$S$ is also arational and $S\sim T$. The second step is to obtain
control over the way that trees can fail to be arational; this is
accomplished in Sections 5 and 6.  Here, we bring a study of standard
geodesics in Outer space, which serve as surrogates for Teichm\"uller
geodesics, and the main result there is Lemma \ref{combined}, which
shows that if $G_t$ is a folding line that converges to a tree $T
\notin \mathcal{AT}$, then the image of $G_t$ for large $t$ in
$\mathcal{F}$ is a uniformly bounded set.  The last two steps involve
running Klarreich's argument and collecting some basic facts about the
point set topology of the spaces $\mathcal{AT}$,
$\mathcal{AT}/\mysim$, and $\partial \mathcal{F}$.  This is the
content of Section 7, where the main result is proved.

Technically, our arguments use the geometry of Outer space and folding
paths as developed in \cite{BF11}, the structure theory of trees in
$\partial \CVN$ developed recently by Coulbois, Hilion, Lustig and
Reynolds \cite{CHL09,CHL08a,CHL08b,CHL08c,CHL07,R10a,R10b,R12,CHR11}. 

\vspace{.2cm}

\noindent {\bf Note}: Very recently, the main result of this paper was also announced by Hamenst\"adt \cite{H12}.

\vspace{.2cm}

\noindent \emph{Acknowledgments: We wish to thank Mark Feighn for
discussions with the first author in which Lemma \ref{L.Vertex} and
Theorem \ref{U-turn3} emerged. We also wish to thank Ursula Hamenst\"adt for
prompting us to finally finish writing this paper.}

\section{Background}

Let $\FN$ denote the free group of rank $N$.  Throughout, we consider
isometric actions of $\FN$ on $\mathbb{R}$-trees; all actions are
assumed minimal.  Let $T$ be a tree; a subset
$I \subseteq T$ is called an \emph{arc} if $I$ is isometric to a segment
in $\mathbb{R}$.  An arc is \emph{non-degenerate} if it contains more than
one point.  For a subset $Y$ of an $\FN$-tree $T$, the stabilizer of $Y$,
denoted $Stab(Y)$, is the set-wise stabilizer of $Y$. In this
section, we collect some definitions and basic results.

\subsection{Outer space and very small $\FN$-trees}

A subset $Y \subseteq T$ that is the convex hull of three points is
called a \emph{tripod} if $Y$ is not a segment.  An action $\FN
\curvearrowright T$ on a tree $T$ is \emph{very small} if for any
non-degenerate arc $I \subseteq T$, either $\Stab(I) = \{1\}$ or
$\Stab(I)$ is a maximal cyclic subgroup of $\FN$, and if for any
tripod $Y \subseteq T$, $Stab(Y)=\{1\}$.  An action $\FN
\curvearrowright T$ is \emph{discrete} (or \emph{simplicial}) if the
$\FN$-orbit of any point of $T$ is a discrete subset of $T$.

The \emph{unprojectivised Outer Space} of rank $N$, denoted $\cvn$, is
the topological space whose underlying set consists of free, minimal,
discrete, isometric actions of $\FN$ on $\mathbb{R}$-trees. For
$T\in\cvn$ we
frequently consider the quotient graph $T/\FN$; this is a marked
metric graph, i.e. there is an identification $\pi_1(T/\FN)\cong\FN$
defined up to conjugation and edges of $T/\FN$ have positive lengths.

A minimal
$\FN$-tree is completely determined by its translation length function
\cite{cm}; this gives an inclusion $\cvn
\subseteq\mathbb{R}^{\FN}$ and a topology on $\cvn$. The non-trivial
points in the closure $\barcvn$ in $\R^{\FN}$ are very small isometric
actions of $\FN$ on $\mathbb{R}$-trees \cite{CL95, BF94}. The
Culler-Vogtmann Outer space, denoted $\CVN$, is the image of $\cvn$ in
the projective space $\mathbb{P}\mathbb{R}^{\FN}$; the points of
$\CVN$ are thought of as free, simplicial $\FN$-trees of co-volume
one. $\CVN$ is canonically a complex of simplices-with-missing-faces
(which we simply call {\it simplices}),
with an (open) simplex corresponding to varying the lengths of edges
(and keeping them positive) on a fixed marked graph. The closure
$\barCVN$ of $\CVN$ in $\mathbb{P}\mathbb{R}^{\FN}$ is compact and
$\partial \CVN=\barCVN-\CVN$ is the projectivization of
$\partial\cvn=\barcvn-\cvn$ and consists of very small trees that are
either non-free or non-simplicial. 
 
The group $Out(\FN)$ acts on $\cvn$, $\barcvn$, $\CVN$ and $\barCVN$:
given a tree $T$ with length function $l_T$ and an element $\Phi \in
Out(\FN)$, for $g \in \FN$, set $l_{T\Phi}(g):=l_T(\phi(g))$, where
$\phi$ is any lift of $\Phi$ to $Aut(\FN)$.

Let $T \in \partial \cvn$, and let $H \leq \FN$ be
finitely generated.  If $H$ does not fix a point in $T$, then we let
$T_H$ stand for the minimal $H$-invariant subtree of $T$; $T_H$ is the
union of axes of hyperbolic elements of $H$.  If $T$ has trivial arc
stabilizers, which is always the case when $T$ has dense orbits, then
for any finitely generated $H \leq \FN$, there is a unique minimal
tree for $H$: either $T_H$ in the case of the previous sentence, or
the unique fixed point of $H$, if $H$ contains no hyperbolic element.

\subsection{Algebraic Laminations and Currents}

We review algebraic laminations associated to $\FN$-trees; see
\cite{CHL08a} and \cite{CHL08b} for details.  Let $\partial \FN$
denote the Gromov boundary of $\FN$ --- \emph{i.e.} the boundary of
any Cayley graph of $\FN$; boundaries of hyperbolic spaces are
reviewed below (equivalently, $\partial \FN$ is the space of ends of
$\FN$).  Let $\partial^2(\FN):=\partial \FN \times \partial
\FN\ssm\Delta$, where $\Delta$ is the diagonal.  The left action of
$\FN$ on a Cayley graph induces actions by homeomorphisms of $\FN$ on
$\partial \FN$ and $\partial^2 \FN$.  Let $i: \partial^2 \FN
\rightarrow \partial^2 \FN$ denote the involution that exchanges the
factors.  A \emph{lamination} is a non-empty, closed, $\FN$-invariant,
$i$-invariant subset $L\subseteq \partial^2 \FN$.

Associated to $T \in \partial \cvn$ is a lamination $L(T)$, which is
constructed as follows.  Let
$$L_{\epsilon}(T):=\overline{\{(g^{-\infty},g^{\infty})|l_T(g)<\epsilon \}}$$
and define
$L(T):= \cap_{\epsilon > 0} L_{\epsilon}$.

Let $T \in \partial \cvn$, and let $H\leq \FN$ be finitely generated.
Then $H$ is virtually a retract of $\FN$ and, hence, is quasi-convex
in $\FN$; so $\partial^2 H$ embeds in $\partial^2 \FN$.  We say that
$H$ \emph{carries} a leaf of $L(T)$ if there is a leaf $l \in L(T)$
such that $l \in \partial^2 H$.  We note that $H$ carries a leaf of
$L(T)$ if and only if either some element of $H$ fixes a point in $T$,
or the action $H \curvearrowright T_H$ is not discrete.

A \emph{(measured geodesic) current} is an $\FN$-invariant Radon
measure $\nu$ on $\partial^2 \FN$, \emph{i.e} $\nu$ is a Borel measure
that is finite on compact subsets of $\partial^2 \FN$.  Let
$Curr(\FN)$ denote the set of currents; equip $Curr(\FN)$ with the
weak* topology.  The group $Out(\FN)$ acts on $Curr(\FN)$ on the left
as follows: let $C \subseteq \partial^2 \FN$ be compact, let $\Phi \in
Out(\FN)$, and let $\nu \in Curr(\FN)$, then
$\Phi(\nu)(C):=\nu(\phi^{-1}(C))$, where $\phi \in Aut(\FN)$ is any
lift of $\Phi$.

If $g \in \FN$ is such that the conjugacy class of $g$ does not
contain an element of the form $h^k$ for $h \in \FN$ and $k>1$, then
there is a \emph{counting current}, denoted $\eta_g$, associated to
the conjugacy class of $g$. We also set $\eta_{g^k}=k\eta_g$ and
frequently write $g$ instead of $\eta_g$.
In \cite{KL09a}, Kapovich and Lustig
establish the following:

\begin{prop}\cite[Theorem A]{KL09a}\label{P.Inter}
 There is a unique $Out(\FN)$-invariant, continuous length pairing
 that is $\mathbb{R}_{\geq 0}$ homogeneous in the first coordinate and
 $\mathbb{R}_{\geq 0}$-linear in the second coordinate
$$
\langle \cdot, \cdot \rangle : \barcvn \times Curr(\FN) \to \mathbb{R}_{\geq 0}
$$
Further, $\langle T, \eta_g \rangle=l_T(g)$ for all $T \in \cvn$ and all counting currents $\eta_g$. 
\end{prop}

The support $Supp(\nu)$ of a current $\nu$ is a lamination
on $\FN$; $Supp(\nu)$ has an isolated point if and only if $\nu$ has
an atom.  Kapovich and Lustig give the following characterization of
zero length:

\begin{prop}\cite[Theorem 1.1]{KL10d}\label{P.ZeroInter}
 Let $T \in \barcvn$, and let $\nu \in Curr(\FN)$.  Then $\langle T,
 \nu \rangle =0$ if and only if $Supp(\nu) \subseteq L(T)$.
\end{prop}

We let $\mathbb{P}Curr(\FN)$ denote the space of projective classes
(\emph{i.e.} homothety classes) of currents.  The action of $Out(\FN)$
on $\mathbb{P}Curr(\FN)$ is not minimal, but there is a unique minset
$\mathbb{P}M_N \subseteq \mathbb{P}Curr(\FN)$ that is the closure of
projective currents corresponding to primitive conjugacy classes of
$\FN$ \cite{Mar97},\cite{Kap06}; let $M_N$ denote the preimage of
$\mathbb{P}M_N$ in $\mathbb{P}Curr(\FN)$.

\subsection{Gromov Hyperbolic Spaces}

We give a very brief review of Gromov hyperbolic spaces and their
boundaries.  Let $(X,d)$ be a metric space, and let $p \in X$ be a
basepoint.  For $x,y \in X$, the \emph{Gromov product} of $x$ and $y$
(relative to $p$) is defined as
$$
(x,y)=(x,y)_p := \frac{1}{2}(d(x,p)+d(y,p) - d(x,y))
$$ 
The metric space $(X,d)$ is called \emph{Gromov hyperbolic} if
there is some $\delta \geq 0$ such that for any $x,y,z \in X$, one has
$$
(x,z) \geq \min\{ (x,y), (y,z) \} - \delta
$$ 
If $(X,d)$ is a geodesic metric space, then hyperbolicity of $(X,d)$ also can be characterized by geodesic triangles being \emph{thin}.  

If $(X,d)$ is Gromov hyperbolic, then one says that a sequence of
points $\{x_n\}$ \emph{converges} if $(x_n, x_m) \to \infty$ as $m,n
\to \infty$.  Two convergent sequences $\{x_n\},\{y_n\}$ are
\emph{equivalent} if $(x_n, y_n) \to \infty$.  The \emph{boundary}
$\partial X$ of $X$ is defined to be the collection of equivalence
classes of convergent sequences in $X$; two equivalence classes of
sequences are close in $\partial X$ if any pair of representatives
have large Gromov product for all large $n$.  That all this is
well-defined follows from hyperbolicity.

Given metric spaces $(X,d)$ and $(X',d')$ and a number $C$, a function $f:X \to X'$ is called a $C$-\emph{quasi-isometric embedding} if for all $x,y \in X$
$$
\frac{1}{C}d(x,y)-C \leq d'(f(x),f(y)) \leq Cd(x,y)+C
$$ 
The map $f$ is a \emph{quasi-isometry} if in addition, for any $z' \in X'$, there is $z \in X$ such that
$$
d'(f(z),z') \leq C
$$

If the spaces $X$ and $X'$ are equipped with an action of a group $G$,
one arrives at the obvious notion of $G$-\emph{equivariant
  quasi-isometry}.  Any quasi-isometry $X \to X'$ between Gromov
hyperbolic spaces induces a homeomorphism $\partial X \to \partial
X'$.

A \emph{quasi-geodesic} in $X$ is a quasi-isometrically embedded copy
of an interval of $\mathbb{R}$.  Two quasi-geodesic rays $r,
r':[0,\infty) \to X$ with $r(0)=r'(0)=0$ are \emph{equivalent} if
  their images have finite Hausdorff distance in $X$.  The boundary
  $\partial X$ coincides with the collection of equivalence classes of
  quasi-geodesic rays (based at $p$), where two classes of rays are
  close if a pair of representatives stay close for a large initial
  segment of $[0,\infty)$.

\subsection{The Complex of Free Factors}

The \emph{complex of free factors}, denoted $\mathcal{F}$, has as
vertices conjugacy classes of non-trivial proper free factors of
$\FN$, where conjugacy classes $[A^1], \ldots, [A^{k+1}]$ span a
simplex in $\mathcal{F}$ if and only if there are representatives
$A^1, \ldots, A^{k+1}$ such that after possibly reordering $A^1 <
\ldots < A^{k+1}$.  Regard $\mathcal{F}$ as a metric space by
identifying each simplex with a standard simplex, and endow the
resulting space with the path metric.  Being its 1-skeleton, the
\emph{graph of free factors} is quasi-isometric to the complex of free
factors. When the rank $N=2$ the complex $\mathcal F$ is a discrete
set, but after a natural modification of the definition it becomes
homeomorphic to the Farey graph. In this paper we will always assume
$N\geq 3$.
We have:

\begin{prop}\cite[Main Theorem]{BF11}
 The metric space $\mathcal{F}$ is hyperbolic.
\end{prop}

Throughout the sequel, we shall use the term \emph{factor} to mean a
conjugacy class of non-trivial, proper free factors of $\FN$;
oftentimes, we will blur the distinction between conjugacy classes and
the subgroups representing them, since we expect little confusion to
arise from this. A conjugacy class of an element or a finitely
generated subgroup of $F_N$ is {\it simple} if it is contained in a
factor.

There is a coarsely well-defined projection $\pi:\cvn \to
\mathcal{F}$: associate to $T \in \cvn$ the collection of factors
represented by subgraphs of $T/\FN$.  It is noted in \cite[Section
  3]{BF11} that $\pi(T)$ has diameter at most 4 and that if the volume
of an immersion representing a factor $F$ in $T/\FN$ is uniformly
bounded, then $d_{\mathcal{F}}(\pi(T),F)$ is uniformly bounded as
well. The projection $\pi$ descends to a projection
$\CVN\to \mathcal F$, also denoted $\pi$.

Given a number $K$, say that a function $\iota:[0,\infty) \to X$ is a
  \emph{reparameterized quasi-geodesic} if there are $0=t_0< t_1<
  \ldots < t_m <\ldots \in [0,\infty)$ such that
    $diam(\iota([t_i,t_{i+1}])\leq K$ and $|i-j| \leq
    d(\iota(t_i),\iota(t_j))+2$.  

\begin{prop}\cite[Corollary 5.5 and Proposition 9.2]{BF11}\label{P.Uniform}
 Let $G_t$ be a geodesic in $\cvn$. Then $\pi(G_t)$ is a
 reparameterized quasi-geodesic with uniform constant.
\end{prop}

Here and throughout, the phrase \emph{uniform constant} is taken to
mean a constant that depend only on $N=Rank(\FN)$.  

\subsection{Geometry of Outer space}

We now review the Lipschitz distance in $\CVN$, optimal maps, train
track structures and folding paths. For more details the reader is
referred to \cite{fm,bers,BF11}. 

A point of $\CVN$ can be thought of as a graph $G$ equipped with a
marking $\pi_1(G)\cong F_N$ and a metric of volume 1. If $G,G'\in\CVN$
there is a canonical homotopy equivalence $G\to G'$ which commutes
with markings. A map $f:G\to G'$ is a {\it difference of markings} if it
belongs to this homotopy class and has constant slope on every
edge. We denote by $\sigma(f)$ the largest slope, i.e. the Lipschitz
constant for $f$, and we put
$$d(G,G')=\inf\log\sigma(f)$$
where $f$ ranges over all difference of markings. This is the {\it
  Lipschitz distance} in $\CVN$; it is not symmetric. 

Any difference of markings map $f:G\to G'$ with
$d(G,G')=\log\sigma(f)$ is called an {\it optimal map} -- these always
exist. If $f:G\to G'$ is an optimal map, the union of all edges on
which the slope of $f$ is $\sigma(f)$ is the {\it tension graph}
$\Delta_f$. Two directions (i.e. half-edges) in $\Delta_f$ based at a
vertex $v$ are {\it equivalent} if $f$ takes both to the same
direction in $G'$. Equivalence classes are {\it gates}.

A {\it train
  track structure} on a finite graph is a collection of equivalence
relations, one on the set of directions at each vertex, such that at
every vertex there are at least two gates. The tension graph may have
vertices with only one gate, but there is always a subgraph
$\Delta\subset\Delta_f$ with an induced train track structure, and in
fact $\Delta=\Delta_g$ for a perturbation $g$ of $f$. 

Let $\Delta$ be a graph with a train track structure.
A turn (i.e. a pair of directions at a vertex) is {\it illegal} if the
two directions are equivalent, otherwise it is legal. A path in
$\Delta$ is legal if all turns it crosses are legal.

Let $f:G\to G'$ be an optimal map.
A loop $\alpha$ in $G$ is a {\it witness} (or it is {\it maximally
  stretched}) if
$l_{G'}(f_*(\alpha))=\sigma(f) l_G(\alpha)$, where $f_*(\alpha)$
is the immersed loop homotopic to $f(\alpha)$. Equivalently, $\alpha$ is 
contained in $\Delta_f$ and it is legal. There is always a witness
that crosses every edge at most twice and crosses at least one edge
exactly once. In particular, such a loop has length $<2$ and it
represents the conjugacy class of a basis element of $F_N$.

Now suppose that $f:G\to G'$ is an optimal map with $\Delta_f=G$ and
with $\geq 2$ gates at every vertex. For this discussion it is
convenient to rescale $G$ so that the slope of $f$ is 1 on every
edge. Thus we are now viewing $G,G'$ as elements of $cv_N$. 

Set $G_0=G$ and for small $t\geq 0$ let $G_t$ be obtained from $G$ by
identifying initial segments of length $t$ within each gate. We have
natural factorizations $G_0\to G_t\to G'$. A path $G_t$, $t\in [0,L]$
in $cv_N$ from $G$ to $G'$ induced by $f$ is a {\it (greedy) folding
  path} (induced by $f$) if $G_0=G$, $G_L=G'$ and for $t\leq t'$ there are maps
$f_{t,t'}:G_t\to G_{t'}$ such that $f_{t,t}=id$, $f_{0,L}=f$ and
$f_{t_2,t_3}f_{t_1,t_2}=f_{t_1,t_3}$, and so that for any $t_0<L$ the
path $G_t$, $t\in [t_0,t_0+\epsilon]$ is obtained as above by
identifying small segments within each gate with the induced maps
$G_{t_0}\to G_t$. We refer to this particular parametrization as the
{\it natural parametrization}. Given $f:G\to G'$ as above, there is a
unique folding path induced by $f$.

The image of a folding path in $\CVN$ is a folding path in $\CVN$,
usually parametrized by arc-length. Every folding path is a geodesic,
i.e. for $t_1<t_2<t_3$ we have
$d(G_{t_1},G_{t_3})=d(G_{t_1},G_{t_2})+d(G_{t_2},G_{t_3})$, but there
are many geodesics that are not folding paths. In fact, not every pair
of points in $\CVN$ can be connected by a folding path. However, there
is always a {\it standard geodesic} joining a given pair of points: it
is a geodesic which is the concatenation of a path inside a simplex
and a folding path.

\section{Laminations and Dendrites}

An $\FN$-tree $T \in \partial \cvn$ is called \emph{indecomposable} if
for any non-degenerate arcs $I, J \subseteq T$, there are
$g_1,\ldots,g_r \in \FN$ such that $I \subseteq g_1J \cup \ldots \cup
g_rJ$ and such that $g_iJ \cap g_{i+1}J$ is non-degenerate.  The goal
of this section is to prove the following maximality condition about
laminations associated to indecomposable trees.

\begin{prop}\label{P.Maximal}
 Let $T \in \partial \cvn$ be indecomposable.  If $U \in \partial \cvn$ satisfies $L(T) \subseteq L(U)$, then $L(T)=L(U)$.
\end{prop}

To prove this fact, we will need to consider actions by homeomorphisms
of $\FN$ on \emph{dendrites}, which are compact, locally connected,
uniquely arcwise connected metrizable spaces, see e.g. \cite{whyburn}.
The connection to actions in $\partial \cvn$ comes from \cite{CHL07}.

The \emph{weak topology}, also called the \emph{observers' topology}
in \cite{CHL07}, on $T$ has as subbasis the collection of directions
(i.e. complementary components) at
points of $T$; let $T_w$ denote $T$ with the weak topology. Let
$\overline T$ be the metric completion of $T$. Then there are two
topologies 
on $\hat T=\overline T \cup \partial T$: the Gromov (metric) topology
and the weak topology $\hat T_w$ defined in the same way as on $T$.
The weak topology is weaker than the metric topology,
and $\hat T_w$ is a dendrite.  It is shown in \cite{CHL07} that if $T$
has dense orbits, then the quotient space $\partial \FN / L(T)$ is
homeomorphic to $\hat T_w$.  There is a natural embedding of $T_w$
into $\hat T_w$; note that $T_w$ is uniquely arcwise connected but is not
compact.  The action of $\FN$ on $T$ induces an action by
homeomorphisms on $\hat T_w$ for which $T_w$ is invariant.

Note that $T_w$ is the subspace consisting of points of $\hat T_w$
that are contained in the interior of an embedded path in $\hat T_w$,
that is, the set of points $x$ of $\hat T_w$ that are separating.
Call the points of $\hat T_w \ssm T_w$ \emph{endpoints}.  Connected
subsets of $\hat T_w$ are path connected.  Since the metric topology
agrees with the weak topology on finite subtrees of $T$, we have that
segments in $T$ are segments in $T_w$, and tripods in $T$ are tripods
in $T_w$.  Hence the action of $\FN$ on the space $T_w$ is very small.
Any segment in $\hat T_w$ with endpoints in $\hat T_w \ssm T_w$ meets
$T_w$ in an open dense sub-segment.  If $T$ is indecomposable, then so
is $T_w$.

\begin{prop}\label{3to1}
Let $p:X\to Y$ be a surjective map between two dendrites.  Assume that:
\begin{enumerate}
\item [(i)] $X=\hat T_w$ for $T\in\partial\cvn$ indecomposable, and
\item [(ii)] $\FN$ acts on $Y$, and $p$ is $\FN$-equivariant.
\end{enumerate}
Then one of the following holds:
\begin{enumerate}
 \item [(a)] $p$ is a homeomorphism,
 \item [(b)] $Y$ is a point, or
 \item [(c)] there is an open interval $Z\subset Y$ such that for every $z\in Z$ we have $|p^{-1}(z)|>2$.
\end{enumerate}
\end{prop}

Before we begin the proof we will make an observation. Assume that
the conclusion of the above proposition fails.  Suppose $[a,b]$ and $[c,d]$ are two
segments in $X$ with $[a,b]\cap [c,d]=[u,v]$ a nondegenerate
segment. 

\begin{claim}
Assume that
the conclusion of the above proposition fails.  Suppose $[a,b]$ and $[c,d]$ are two segments in $X$ with $[a,b]\cap [c,d]=[u,v]$ a nondegenerate segment. If $p(a)=p(b)$ and if $p(c)=p(d)$, then $p(u)=p(v)$.
\end{claim}

The proof of the claim uses only that dendrites are uniquely arcwise connected.

\begin{proof}

First note that up to symmetry, there are 3 possible configurations, which are shown in Figure
\ref{4points}. 

\begin{figure}[h]
\includegraphics[scale=0.75]{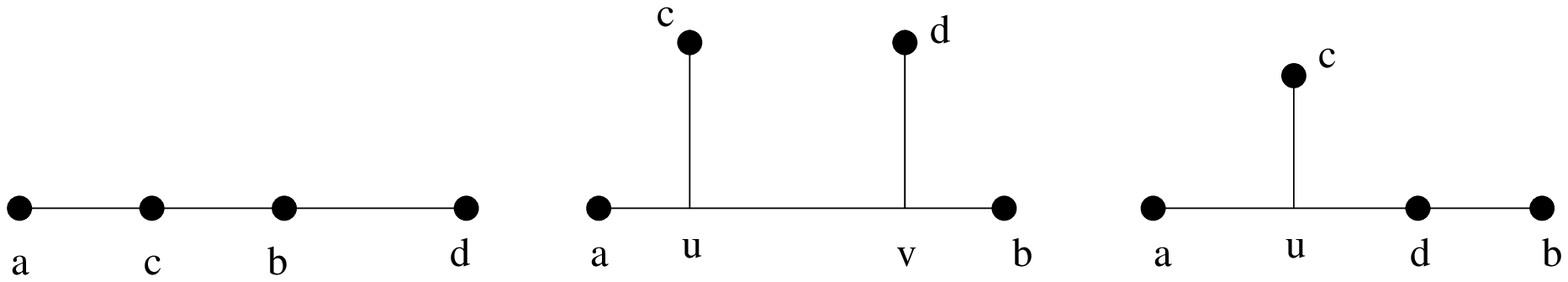}
\caption{}
\label{4points}
\end{figure}

We will contradict the assumption that Proposition \ref{3to1} fails.

For the first configuration, if $p(a)=p(b)=r$ and $p(c)=p(d)=s$ but
$r\neq s$, then take for $Z$ the open interval $(r,s)$. Every $z\in Z$
has a preimage point in each interval $(a,c),(c,b),(b,d)$, a contradiction.

In the second configuration, if $p(u)\neq p(v)$, take $Z=(p(u),p(v))$,
and notice that the preimage of each $z\in Z$ intersects $[a,b]$ in at
least two points, and $[c,u]\cup [d,v]$ in at least one point, a contradiction. 

In the last configuration, if $p(u)\neq p(d)=p(c)$ take
$Z=(p(u),p(d))$.  Each $z\in Z$ has at least two preimages in $(a,b)$
and at least one in $(c,u)$, a contradiction.
\end{proof}

\begin{proof}[Proof of Proposition \ref{3to1}]
First note that if $p$ collapses a nondegenerate segment, then
indecomposability of $X$ and equivariance forces $p$ to be constant, implying that $Y$ is a
point.  So assume that $p$ does not collapse any non-degenerate
interval and that $p$ is not a homeomorphism. This gives that $p$ is
not injective, so there are distinct $a,b\in X$ with
$p(a)=p(b)$. By the pidgeon hole principle, after replacing $[a,b]$ by a smaller interval, we may
assume that $a,b\in T$ and that $p(a)=p(b)$ has valence 2; indeed, since $Y$ has a countable basis, $Y$ contains at most countably
many points with valence $>2$ (this is a theorem of Whyburn).

Again by the pidgeon hole principle, there are distinct $c,d\in [a,b]$ with $p(c)=p(d)$ and with the
$T$-distance between $c,d$ arbitrarily small. Apply indecomposability
to $I=[a,b]$ and $J=[c,d]$ to deduce that $I\subset\cup g_i(J)$ for
$i=1,2,\cdots,k$ and $g_i(J)\cap g_{i+1}(J)$ is nondegenerate. We may
also assume that $k$ is minimal, so in particular $g_1(J)$ and
$g_k(J)$ will contain the endpoints of $[a,b]$. Apply the Claim
to the segments $I$ and $g_i(J)$ (by equivariance, the endpoints of
$g_i(J)$ are mapped to the same point). Thus the endpoints of
$g_i(J)\cap I$ map to the same point $y_i$ in $Y$. We now claim that
$p(a)=p(b)=y_1=\cdots=y_k$. It is clear that $p(a)=y_1$ since $a$ is
an endpoint of $g_1(J)\cap I$ (up to switching $a$ and $b$) and
similarly $p(b)=y_k$. To see that $y_1=y_2$ apply the Claim to
$g_1(J)$ and $g_2(J)$ etc.

We now have points $a=t_0<t_1<\cdots<t_m=b$ in $[a,b]$ with $p(t_i)=y$
for every $i$. We may take $m$ as large as we want by making $J$
small. The images of the intervals $[t_i,t_{i+1}]$ are dendrites
$D_i$ containing $y$, and since the valence of $y$ is 2, as soon as we
have $m\geq 3$ two of the dendrites, say $D_i$ and $D_j$, will have
nondegenerate overlap. Take $Z$ to be an open interval in the
overlap. Then any point in $Z$ will have at least two preimages in
$[t_i,t_{i+1}]$ and at least two in $[t_j,t_{j+1}]$.
\end{proof}

We are now in position to prove Proposition \ref{P.Maximal}.  For the
proof we will need the main result of \cite{CH10}; we note that if $T
\in \partial \cvn$ has dense orbits, then the map $Q$ used to define
$Q$-\emph{index} in \cite{CH10} is the quotient map $Q=Q_T:\partial
\FN \to \partial \FN/L(T)=\hat{T}_w$. Here is a simplified version of
the result of Coulbois-Hilion:

\begin{prop}\label{P.IndexTheorem}\cite[Theorem 5.3]{CH10}
 Let $T$ be a very small $\FN$-tree with dense orbits, and let
 $Q=Q_T:\partial \FN \to \partial \FN/L(T)=\hat{T}_w$. Then there are
at most countably many points $z \in \hat{T}_w$ for which
 $|Q^{-1}(z)|>2$.
\end{prop}

\begin{remark}
 If $T$ in the statement of Proposition \ref{P.Maximal} is of
 \emph{pseudo-surface type} (defined in \cite{CH10}), then the
 statement follows immediately from \cite[Theorem 5.10]{CH10}, and for
 geometric trees the statement follows from the quasi-isometric
 classification of leaves. Our proof of Proposition \ref{P.Maximal}
 is new in the 
 indecomposable \emph{pseudo-Levitt} case.
\end{remark}

\begin{proof}[Proof of Proposition \ref{P.Maximal}]
 Suppose that $U \in \partial \cvn$ satisfies $L(T) \subseteq L(U)$.
 It follows from \cite{Lev94} that $U$ can be assumed to have dense
 orbits; indeed, if $U$ does not have dense orbits, then we can
 collapse the simplicial part of $U$ to get a tree with dense orbits,
 and one easily sees from the definition of $L(\cdot)$ that the
 associated lamination can only be enlarged; see \cite{Lev94} or
 \cite{R12} for details.  One has that the quotient map $\partial \FN
 \to \partial \FN/L(U) =\hat U_w$ factors through $\partial \FN \to
 \partial \FN/L(T)=\hat T_w$, so we get a surjective map $p:\hat T_w
 \to \hat U_w$, which is $\FN$-equivariant.

 Now apply Proposition \ref{3to1}.
 Since $U$ contains more than one point, conclusion (b) is not
 possible.  If conclusion (c) holds, then there are uncountably many
 points of $\hat U$ whose pre-image in $\partial \FN$ contains
 strictly more than two points; but this is impossible by Proposition \ref{P.IndexTheorem}.  Hence, $p$ is a homeomorphism, so $L(T)=L(U)$.
\end{proof}

\section{Arational Trees}
We recall a notion of reduction for very small
trees, introduced in \cite{R12}.  For $T \in \partial \cvn$ and $F$ a
factor, say that $F$ \emph{reduces} $T$ if $F$ acts with dense orbits
on some subtree $Y \subseteq T$.  It should be
emphasized that $Y$ can consist of a single point.  If $Y$ contains
two points, then $Y$ necessarily has infinite diameter, and in this
case the minimal subtree $T_F$ for $F$ is dense in $Y$.  

Use $\mathcal{R}(T)$ to denote the set of all factors
reducing $T$.  It is noted in \cite{R12} that if $F'$ is a factor
carrying a leaf of $L(T)$, then there is $F \in \mathcal{R}(T)$ with
$F \leq F'$; so, regarded as subsets of $\mathcal{F}$,
$\mathcal{R}(T)$ is 1-dense in the set of all factors carrying a leaf
of $L(T)$.

When $\mathcal{R}(T)=\emptyset$, the tree $T$ is arational;
this is equivalent to the statement that no leaf of $L(T)$ is carried
by a factor by \cite{R12}. Toward establishing an intuitive
analogy with surfaces, we note that the analogous laminations are
precisely the arational laminations--\emph{i.e.} the minimal and
filling laminations.

We have the following classification of arational trees:

\begin{prop}\cite[Theorem 1.1]{R12}
 Let $T \in \partial \cvn$.  The following are equivalent:
 \begin{enumerate}
  \item [(i)] $T$ is arational,
  \item [(ii)] $T$ is indecomposable, and if $T$ is not free, then $T$
    is dual to an arational measured lamination on a surface with one
    boundary component.
 \end{enumerate}
\end{prop}

When $T$ is geometric one can prove the equivalence of (i) and (ii)
using the methods of Section 6, see particularly the proof of Lemma
\ref{alternative}. When $T$ is non-geometric, one forms a geometric
resolution and a folding path converging to $T$. In general, under
folding parts of the approximating graphs become geometric. When $T$
is also arational, the folding path gives a sequence of strong
approximations without forward invariant subgraphs whose unrescaled
volumes go to 0, and $T$ is free and indecomposable (see the proof of
Lemma \ref{combined} where a similar argument is used).

If $X \subseteq \partial^2 \FN$, say that a leaf $l=(x,y) \in
\partial^2 \FN$ is \emph{diagonal} over $X$ if there are leaves
$(x_1,x_2),(x_2,x_3), \ldots,(x_{r-1},x_r) \in X$, such that $x=x_1$
and $y=x_r$; and say that $X$ is \emph{diagonally closed} if every
leaf that is diagonal over $X$ belongs to $X$.  Laminations associated
to trees are always diagonally closed \cite{CHL08b}.

We collect the following information about laminations associated to
arational trees; for the statement, $L'(T)$ denotes the
Cantor-Bendixson derivative of $L(T)$, \emph{i.e.} $L'(T)=(L(T))'$ is
the set of non-isolated points of $L(T)$, and we set
$L''(T)=(L'(T))'$ and $L'''(T)=(L''(T))'$.

\begin{prop}\label{P.MinimalLam}
 Let $T \in \partial \cvn$.  
\begin{enumerate}
\item [(i)]
If $T$ is free and indecomposable, then $L'(T)$ is minimal, no leaf of
$L'(T)$ is carried by a factor, and $L(T)$ is obtained from
$L'(T)$ by adding finitely many $\FN$-orbits of isolated leaves, each
of which is diagonal and not periodic.
\item [(ii)] If $T$ is dual to an arational measured lamination on a
  surface with one boundary component, then $L'''(T)$ is minimal, no
  leaf of $L(T)$ is carried by a factor, and $L(T)$ is the
  smallest diagonally closed lamination containing
  $L'''(T)$. 
\end{enumerate}
\end{prop}

\begin{proof}
 Statement (i) follows from the main results of \cite{R10a} and
 \cite{CHR11}.  Statement (ii) follows from straightforward
 considerations about surface laminations and foliations, and the
 reader is assumed to have some familiarity with this; see \cite{CB88}
 and \cite{CHL08b} for background and the next paragraph for details.

 Let $S$ be a hyperbolic surface with one totally geodesic boundary
 component, equipped with a filling, minimal geodesic measured
 lamination $\Lambda$ to which $T$ is dual.  The universal cover
 $\tilde S$ can be identified with a closed convex subset of $\mathbb
 H^2$ whose boundary consists of lines covering $\partial S$ and the
 set of ends of $\tilde S$ is $\partial F_N$. By construction, the
 lift $\tilde \Lambda$ of $\Lambda$ gives a subset of the dual
 lamination $L(T)$. Now $\tilde\Lambda$ is not diagonally closed, and
 we will enlarge it by adding diagonal leaves. A complementary
 component of $\tilde\Lambda$ containing a boundary component of
 $\tilde S$ is the universal cover of the {\it crown set} (see
 \cite{CB88}), i.e. it is obtained from a hyperbolic half-plane
 invariant under a hyperbolic isometry $g$ (the deck transformation
 preserving the boundary) by deleting a pairwise disjoint
 collection of hyperbolic half-planes invariant under $g$, with
 adjacent half-spaces cobounding cusps. Thus the
 set of ends of this region can be identified with
 $\{-\infty\}\cup\Z\cup\{+\infty\}$ where $\pm\infty$ correspond to
 the ends of the boundary component and elements of $\Z$ to the cusps.

Since $L(T)$ is closed and diagonally closed, we see that it contains
$F_N$-orbits of leaves connecting any pair of distinct points in
$\{-\infty\}\cup\Z\cup\{+\infty\}$. To see that $L(T)$ contains no
additional leaves, observe that any biinfinite geodesic that does not
belong to the collection just described must intersect $\tilde\Lambda$
transversally, so it gets nonzero measure and does not belong to
$L(T)$. Finally, we have that $L'(T)$ is obtained from $L(T)$ by
deleting isolated leaves, and these are in the orbit of leaves
connecting $m$ and $n$ for $|m-n|>1$, then $L''(T)$ is obtained from
$L'(T)$ by further deleting orbits of leaves connecting $n$ to
$\pm\infty$, and lastly, $L'''(T)$ is obtained by deleting the orbit
of the boundary.

To see that no leaf is carried by a factor $A$ consider the lift of
$\Lambda$ and the added leaves to the $A$-cover of $S$ and restrict to
the convex core; here, the lifts of leaves of $\Lambda$ must be compact, \emph{c.f.} \cite{R10a}.
\end{proof}

Hence, we get the following:

\begin{cor}\label{no fold}
  Let $T,U \in \partial \cvn$.  If $T$ is arational and if $L'''(T)
  \subseteq L(U)$, then $L(T)=L(U)$.
 \end{cor}

 \begin{proof}
 Using Proposition \ref{P.MinimalLam}, we get that $L'''(T) \subseteq
 L(U)$ implies that $L(T) \subseteq L(U)$.  Apply Proposition
 \ref{P.Maximal} to conclude.
 \end{proof}

For any $T\in \barCVN$ we define
$$T^*=\{\mu\in\mathbb PM_N\mid \<T,\mu\>=0\}$$
Thus if $T\in \CVN$ then $T^*=\emptyset$. An elementary limiting
argument shows that if $T\in\partial \CVN$ then $T^*\neq\emptyset$,
but for efficiency of the exposition this is postponed to Remark
\ref{nonempty} below.

\begin{thm}\label{T.UniqueLam}
 Let $T \in \partial \cvn$.  If $T$ is arational, and if $\mu \in M_N$
 satisfies $\langle T,\mu \rangle = 0$, then $Supp(\mu)=L'''(T)$.  In
 particular, if $U$ is another very small tree satisfying $\langle
 U,\mu \rangle =0$, then $U$ is arational, $L(T)=L(U)$ and $T^*=U^*$.
\end{thm}

\begin{proof}
 If $\mu \in M_N$ satisfies $\langle T,\mu \rangle=0$, then
 Proposition \ref{P.ZeroInter} gives that $Supp(\mu) \subseteq
 L(T)$. The support of a current cannot contain non-periodic isolated
 leaves, since translates of such leaves have accumulation points (recall that currents are Radon measures).

 We now show that the support of a current in $M_N$ cannot contain a
 periodic leaf corresponding to a conjugacy class not carried by a
 factor. Here we
 will use the fact that a conjugacy class $g$ is not simple if and
 only if there is a basis with respect to which the Whitehead graph of
 $g$ is connected and has no cut points.

Choose a basis for $\FN$, and let $\mu_n$ be a sequence
 of currents corresponding to primitive conjugacy classes $g_n$ with
 $\mu_n$ converging to $\mu$.  Each $g_n$ has Whitehead graph that
 either is disconnected or has a cut point.  After passing to a
 subsequence, $g_n$ all have the same Whitehead graph $W$.  If $l$ is
 a periodic leaf of $Supp(\mu)$ corresponding to $g \in \FN$, then for
 $n >>0$, $g_n$ contains $g^2$ as a subword; it follows that the
 Whitehead graph for $g$ is contained in $W$, so $g$ is carried by a
 factor (since the statement about Whithead graphs is true for
 every basis). We conclude by Proposition \ref{P.MinimalLam} that
 $Supp(\mu)\subseteq L'''(T)$, and hence $Supp(\mu)=L'''(T)$. 

 Now let $U \in \partial \cvn$ be some other tree such that $\langle
 U,\mu \rangle =0$.  Again by Proposition \ref{P.ZeroInter}, we have
 that $L'''(T) \subseteq L(U)$.  Apply Corollary \ref{no fold} to
 conclude.
\end{proof}

\begin{cor}\label{stick}
Let $S, T \in \partial cv_N$.  If $S^* \subseteq T^*$ and
$\mathcal{R}(S)\neq\emptyset$, then $\mathcal{R}(S) \cap
\mathcal{R}(T) \neq \emptyset$.
\end{cor}

\begin{proof}
First suppose that there is a factor $F$ fixing a point in $S$.  Choose a
basis for $F$, and note that $\{\eta_g|g$ has length $\leq 2\} \subseteq
S^* \subseteq T^*$, so by Serre's Lemma $F$ fixes a point in $T$ as well.

If no factor fixes a point in $S$, then one finds $F \in \mathcal{R}(S)$
such that $F \curvearrowright S_F$ is arational (e.g. choose $F$ of
minimal rank).  Let $\mu$ be a current
that is supported on $L'''(S_F)$; we have that $Supp(\mu)$ fills $F$ and
that $\mu \in S^* \subseteq T^*$.  Hence, by Proposition \ref{P.ZeroInter}
$L'''(S_F) \subseteq L(T)$, so by Theorem \ref{T.UniqueLam}
either $F$ acts arationally on $T_F$ or else $F$
fixes a point in $T$; in either case, $F \in \mathcal{R}(T)$.
\end{proof}

We obtain the following result, which suggests that arational trees
lie ``at infinity'' with respect to $\mathcal{F}$.

\begin{cor}\label{C.ArationalAtInfinity}
Let $T_n \in \CVN$ be a sequence of trees converging to an arational
tree $T$, and let $Y_n=\pi(T_n)$ denote a projection to $\mathcal{F}$.
For any basepoint $0 \in \mathcal{F}$, we have $d(0,Y_n) \to \infty$.
\end{cor}

We follow Feng Luo's argument, an adaptation of \cite{kobayashi}.

\begin{proof}
We may assume that $Y_n$ is a factor generated by a uniformly bounded
loop in $T_n/F_N$ representing a conjugacy class $g_n$. This
guarantees that $Y_n$ is at uniformly bounded distance from
$\pi(T_n)$. Choose a basepoint $0 \in \mathcal{F}$. Toward
contradiction, suppose that $d(0,Y_n)$ does not go to infinity; by
passing to a subsequence, we can assume that $d(0,Y_n)=r$ for every
$n$.  Then there are paths $0=A_n^0, A_n^1, \ldots, A_n^{r-1},
A_n^{r}=Y_n \in \mathcal{F}$. 

Choose simplicial trees $T_n^i$ in which
$A_n^i$ is elliptic and let $g_n^i$ be conjugate into both $A_n^i$ and
$A_n^{i-1}$ for $i=1,\cdots,r$; also arrange that $g_n^r=g_n$ and that
$T_n^0$ does not depend on $n$. After possibly passing to a further
subsequence and rescaling, we can assume that $\lim
T_n^i=T^i\in\overline{cv}_N$ and $\lim g_n^i=\eta^i\in\mathbb PM_N$
for every $i$. By the Kapovich-Lustig continuity we have
$\<T^i,\eta^i\>=0=\<T^{i-1},\eta^i\>$ so by induction on $i$ and using
Theorem \ref{T.UniqueLam} we see that none of $T^i$'s are arational and in
particular $\eta^r$ does not have length 0 in any arational tree.

To get a contradiction, we argue $\<T,\eta^r\>=0$. By construction,
$\<T_n,g_n\>\leq C$ so $\<T_n/\mu_n,g_n/|g_n|\>\leq C/\mu_n|g_n|$ where
$|g_n|$ is the length of $g_n$ in a fixed rose, and $\mu_n$ are
rescaling constants so that $T_n/\mu_n\to T$. Again by the continuity
of the length pairing and the fact that both $|g_n|$ and
$\mu_n$ are bounded below, it suffices to argue that
$|g_n|\to\infty$, i.e. that it is possible to choose $g_n$'s to be all
distinct after a subsequence. If not, then after a subsequence all
$T_n$ belong to the same simplex and the limit $T$ is a simplicial
tree with trivial edge stabilizers, so certainly not arational.
\end{proof}

\begin{remark}\label{nonempty}
A similar argument shows that if $T\in\partial\CVN$ then
$T^*\neq\emptyset$. If $T$ is in the closure of a simplex there are
simple elliptic elements, and otherwise construct $\mu\in T^*$ as the
limit of $g_n/|g_n|$ as above.
\end{remark}

\section{Primitive Elements and Vertex Groups}

By a \emph{vertex group} we mean a vertex stabilizer in a very small
simplicial $\FN$-tree.  We associate to $g \in \FN$, respectively $A
\leq \FN$, the smallest free factor containing it, which we denote by
$Fill(g)$, respectively $Fill(A)$; $g$ ($A$) is \emph{simple} if
$Fill(g)$ ($Fill(A)$) is a proper subgroup of $\FN$, hence we get a
map $Fill: \{non-trivial$ $simple$ $elements$ $(subgroups)\} \to \F$.

\begin{lemma}\label{L.Vertex}
There is a constant $C$ such that for every very small simplicial
$\FN$-tree $T$, the set of simple elements fixing a point of $T$ map
under $Fill$ to a set of diameter at most $C$ in $\mathcal{F}$.
\end{lemma}

We assume that the reader is familiar with Whitehead's algorithm
\cite{LS77}. The argument is an adaptation of the proof of \cite[Lemma
  3.2]{BF11}.   

\begin{proof}
Let $T \in \partial \cvn$ be simplicial.  If $T$ has an edge $e$ with
trivial stabilizer, then collapsing every edge not in the orbit of $e$
gives a tree $T'$ corresponding to a free splitting of $\FN$, and
every simple elliptic element of $T$ is elliptic in $T'$.  The image
under $Fill$ of the simple elliptic elements of $T'$ has diameter at
most $2$ in $\mathcal{F}$.

So, assume that $T$ has no edge with trivial stabilizer; collapse
edges outside of a fixed orbit of edges and replace $T$ with the
resulting $1$-edge splitting.  This increases the diameter of the
image of $Fill$ by at most $2$.  We have two cases to consider,
corresponding to whether $T/\FN$ is a segment or a loop.

First suppose that $T/\FN$ is a segment, so $T$ corresponds to a
splitting $A\ast_w B$.  By Swarup's theorem \cite{swarup} (see also 
\cite[Lemma 4.1]{BF94}), we have, possibly after interchanging $A$
and $B$, that
$A=\langle a_1, \ldots,a_k,w \rangle$ and $B=\langle
b_1,\ldots,b_l\rangle$, where $\{a_1,\ldots,a_k,b_1,\ldots,b_l\}$ is a
basis for $\FN$ and $w\in B$.

If $w$ is contained in a factor $B'$ of $B$, then $A$ is
contained in the factor $\<a_1,\cdots,a_k\>*B'<F_N$ and the Lemma
follows. So assume that $w$ fills $B$, and after possibly changing the
basis of $B$, that the Whitehead graph of $w$ with respect to the
basis $\{b_i\}$ is connected and has no cut points. 

Now let $g$ be a simple conjugacy class in $A$, i.e. a cyclic word in
the $a_i$'s and $w$. If $w$ does not appear, the image of $Fill(g)$ is
at distance $\leq 1$ from $\<a_1,\cdots,a_k\>$ and we are done. 
Likewise, if some $a_i$
does not appear, then $Fill(g)$ is at distance $\leq 1$ from
$\<a_1,\cdots,\hat a_i,\cdots,a_k\>*B$, hence at distance $\leq 4$
from $\<a_1,\cdots,a_k\>$, and, again, we are done.

We now apply the
Whitehead algorithm that transforms $g$ to a cyclic word that does not
involve all the basis elements. At each step a Whitehead automorphism is
applied whose effect on $g$ is that it gets shorter. The Whitehead
automorphism can be read off from the Whitehead graph of $g$ in the basis
$\{a_i\}\cup\{b_j\}$, which
has a cut vertex
called the {\it special letter}.
Note that the Whitehead graph of $g$ contains as
a subgraph the Whitehead graph of $w$ with one edge removed (since $w$
doesn't get ``closed up'' in $g$) and this subgraph is connected by
our assumption on $w$.

If the special letter is $a_i^{\pm 1}$, then all $b_j^{\pm 1}$'s are
on one side of the cut vertex, and therefore $B$ is either fixed or
gets conjugated by $a_i^{\pm 1}$. The word $w$ inside $g$ stays
unaffected.

Now assume the special letter is $b_i^{\pm 1}$. Say $w=xw_1y$ as a
word in $\{b_j\}$. Thus $x\neq y^{-1}$, and in the Whitehead graph $W$ of
$g$ with respect to $\{a_i\}\cup\{b_j\}$, $b_i^{\pm}$ is a cut
vertex. The Whitehead graph of $g$ with respect to
$\{a_1,\cdots,a_k,w\}$ is obtained from $W$ by removing all vertices
$b_j^{\pm 1}$ except for $x$ and $y^{-1}$ and all edges incident to
them, and renaming $x$ to $w$ and $y^{-1}$ to $w^{-1}$. Therefore this
Whitehead graph is disconnected and $g$ is contained in a factor
$C$ complementary to $B$, so the Lemma again follows. (The factor
$C$ is
obtained from $\<a_1,\cdots,a_k\>$ by applying the Whitehead
automorphism with special letter $w$. Topologically, one can blow up
the rose on $a_1,\cdots,a_k,w$ by inserting an ideal edge which is not
crossed by $g$; then in the graph of spaces corresponding to $A*_wB$
collapse the 2-cell from this ideal edge to get a graph representing
$F_N$, containing the rose on $b_1,\cdots,b_l$ and a representative of
$g$ disjoint from this rose.)

To summarize, the Whitehead algorithm runs as long as some $a_i^{\pm
  1}$ is a special letter and during this time $B$ is fixed up to
conjugacy. When some $a_i$ is completely erased from $g$ or some
$b_j^{\pm 1}$ becomes the
special letter, we are done by the above discussion.

When $T/F_N$ is a loop the argument is similar: we can write the
vertex group as $\<a_1,\cdots,a_{N-1},w^c\>$ where
$F_N=\<a_1,\cdots,a_{N-1},c\>$ and $w\in
\<a_1,\cdots,a_{N-1}\>$. Again we may assume that the Whitehead graph
of $w$ in $\<a_1,\cdots,a_{N-1}\>$ is connected and has no cut
points. Let $g$ be simple and written as a cyclic word in
$a_1,\cdots,a_{N-1},w^c$. We may assume it involves all of these
generators. If $w=xw_1y$ then the Whitehead graph of $w$ in the basis
$\{a_1,\cdots,a_{N-1},c\}$ is obtained from the Whitehad graph of $w$
in $\{a_1,\cdots,a_{N-1}\}$ by removing an edge joining $y^{-1}$ and
$x$, adding edges from $c^{-1}$ to $x$ and from $c^{-1}$ to $y^{-1}$
and perhaps adding more edges. In particular, the subgraph spanned by
the $a_i^{\pm 1}$ and by $c^{-1}$ is connected and has no cut
points. Since $c$ and $c^{-1}$ are not connected to each other,
neither $c$ nor $c^{-1}$ can be cut points. If say $a_i$ is a cut
point, $c$ is the only vertex on one side and the associated
automorphism is of the form $c\mapsto ca_i$ and all $a_j$ fixed. It
follows that $a_i$ is either $x^{-1}$ or $y$; either way the
automorphism preserves $\<a_1,\cdots,a_{N-1}\>$ and conjugates
$w$. The proof concludes as before.
\end{proof}

To extend Lemma \ref{L.Vertex} to all trees in $\partial \cvn$, we use
the following result.  Recall that for $T \in \partial \cvn$,
$\mathcal{R}(T)$ denotes the collection of all factors reducing $T$.

\begin{prop}\cite[Theorem 1.3]{R12}
 Let $T \in \partial \cvn$, and assume that $T$ is not arational.
 There is a simplicial tree $T_0$ such that for any $F \in
 \mathcal{R}(T)$, some element of $F$ fixes a point in $T_0$.
\end{prop}
  
It follows that the diameter of $\mathcal{R}(T)$ in $\mathcal{F}$ is
at most two more than the diameter of the $Fill$-image of the set of
simple elliptic elements in $T_0$, hence we get:

\begin{cor}\label{C.Bounded}
 Let $T \in \partial \cvn$, and assume that $T$ is not arational.  The
 set $\mathcal{R}(T)$ has uniformly bounded diameter in $\mathcal{F}$.
\end{cor}

\section{Sequences of Geodesics}

In this section we examine possible accumulation sets of sequences of
geodesics in $\CVN$. The main result is Theorem \ref{U-turn3}.

\subsection{Limits of Sequences of Geodesics}

Fix a basis $\mathcal{B}$ for $\FN$; for $g \in \FN$, let $|g|$ denote
the word length of $g$ in $\mathcal{B}$.  We work in Outer space with
graphs normalized to have volume 1, and we sometimes consider
universal covers.
The following is essential for the remainder
of the paper:

\begin{remark}
All (projectivized) currents come from ($\mathbb{P}M_N$) $M_N$.  
\end{remark}

\noindent For a tree $T \in \overline{\cvn}$ and for $g \in \FN$, we
use $\langle T, g \rangle$ to mean $\langle T, \eta_g \rangle$, which
is the translation length of $g$ in $T$ by Proposition \ref{P.Inter}.
We will use Proposition \ref{P.Inter} below without reference.

Suppose that we have a sequence of geodesics $[S_n,T_n]$ in $\CVN$. We
assume that $d(S_n,T_n)=\log\lambda_n$, that $T_n/\mu_n$ converges to
$T$ and that $S_n/\kappa_n$ converges to $S$ for some
$\lambda_n,\mu_n,\kappa_n \geq 1$.

\begin{lemma}\label{bounded below}
In this situation, $\inf \frac{\kappa_n\lambda_n}{\mu_n} >0$.
\end{lemma}

\begin{proof}
Fix some $g\in \FN$. Then $\lambda_n\langle S_n,g \rangle \geq \langle
T_n,g \rangle$ i.e. $$\frac{\kappa_n\lambda_n}{\mu_n} \langle S_n/\kappa_n,g
\rangle \geq \langle T_n/\mu_n,g\rangle.$$ 
Passing to the limit and assuming
$\kappa_n\lambda_n/\mu_n\to 0$ gives $\langle T,g\rangle=0$, which is
impossible, since $g$ was arbitrary.
\end{proof}

\begin{lemma}\label{bounded above}
If $\frac{\kappa_n\lambda}{\mu_n}$ is bounded above then $T^*\supseteq S^*$.
\end{lemma}

\begin{proof}
We may assume $\kappa_n\lambda_n/\mu_n=1$. As above, for any $g\in \FN$
we have $\langle S,g \rangle \geq \langle T,g \rangle$.  Let $\nu$ be
some current, and let $g_n \in \FN$ be such that $g_n/|g_n| \to \nu$.
Passing to the limit gives that $\langle S,\nu \rangle =0$ implies
that $\langle T,\nu \rangle =0$.
\end{proof}

The same line of reasoning shows:

\begin{lemma}\label{L.LipschitzLam}
 Let $S, T \in \barcvn$.  If there is a Lipschitz map $S \to T$, then
 $S^* \subseteq T^*$.
\end{lemma}

To the sequence of geodesics $[S_n,T_n]$ we associate a closed subset of
projectivized measured currents $\mathcal C([S_n,T_n])$ defined as
the set of (projective classes of) those $\nu$ that can be represented
as $\lim \gamma_n/|\gamma_n|$ where $\gamma_n$ is a maximally
stretched simple loop in
$S_n$ (i.e. a legal simple loop in the tension graph for some optimal map
$S_n\to T_n$, see Section 2.5).  Without further comment, we always allow
passing to a subsequence of $[S_n,T_n]$.

\begin{lemma}\label{currents}
Let $[S_n,T_n]$ be geodesics such that $d(S_n,T_n)=\log\lambda_n$,
$T_n/\mu_n\to T$ and such that $S_n/\kappa_n\to S$.
If $\kappa_n\lambda_n/\mu_n\to\infty$, then $\mathcal C([S_n,T_n])
\subseteq S^*$.
\end{lemma}

\begin{proof}
Since $\gamma_n$ is legal we have $\lambda_n\langle S_n,\gamma_n
\rangle =\langle T_n,\gamma_n \rangle$ i.e.
$$\langle S_n/\kappa_n,\gamma_n/|\gamma_n|\rangle =\frac
{\mu_n}{\lambda_n\kappa_n}\langle T_n/\mu_n,\gamma_n/|\gamma_n| \rangle$$ 
On the other hand, continuity of $\langle \cdot, \cdot \rangle$ gives $\langle T_n/\mu_n,\gamma_n/|\gamma_n| \rangle \to \langle T,\nu \rangle < \infty$, so 
$\langle S,\nu \rangle=0$.
\end{proof}

The following result summarizes the previous lemmas.

\begin{thm}\label{U-turn3}
Let $[S_n,T_n]$ be a sequence of geodesics, $U_n\in [S_n,T_n]$,
and assume $S_n/\kappa_n\to S$, $U_n/\rho_n\to
U$, $T_n/\mu_n\to T$. Then
\begin{enumerate}[(i)]
\item If $\rho_n e^{d(U_n,T_n)}/\mu_n$ is bounded, $T^*\supseteq U^*$; in
  particular, $T$ is not free simplicial if $U$ is not.
\item If $\rho_n e^{d(U_n,T_n)}/\mu_n$ is not bounded, then $S^* \cap
  U^* \neq \emptyset$.
\end{enumerate}
\end{thm}

\begin{proof}
The conclusion in (i) follows from Lemma \ref{bounded above}. In case
(ii) we may take $\rho_n e^{d(U_n,T_n)}/\mu_n\to\infty$ and therefore
$\kappa_n e^{d(S_n,T_n)}/\mu_n\to\infty$ since the ratio $\kappa_n
e^{d(S_n,U_n)}/\rho_n$ is bounded below by Lemma
\ref{bounded below}. Then observe that
$\mathcal C([S_n,T_n])\subseteq \mathcal C([U_n,T_n])$ and by Lemma
\ref{currents} $\mathcal C([S_n,T_n])\subseteq S^*$ and $\mathcal
C([U_n,T_n]) \subseteq U^*$. Since $\mathcal
C([S_n,T_n])\neq\emptyset$, we have $S^*\cap U^*\neq\emptyset$.
\end{proof}

\begin{cor}\label{defined and continuous}
 Suppose that $S_n/\kappa_n$ converges to $S$, that $T_n/\mu_n$
 converges to $T$, and let $[S_n,T_n]$ be a geodesic.
 If $S^*=T^*$, then any tree $U$ representing a point in the
 accumulation set of $[S_n,T_n]$ in $\overline{CV}_N$ satisfies
 $U^*\cap T^*\neq\emptyset$. In particular, if $S$ and $T$ are
 arational, then so is $U$, and $U^*=S^*=T^*$.
\end{cor}

\begin{cor}\label{injective}
 Suppose $S_n/\lambda_n \rightarrow S$, $T_n \rightarrow T$, and let
 $[S_n,T_n]$ be a geodesic.  If $S$ and $T$ are arational with
 $S^*\neq T^*$, then the accumulation set of $[S_n,T_n]$ in $\barCVN$
 includes points of $\CVN$.
\end{cor}

\begin{proof}
The accumulation set is connected, it includes $S,T$ and every tree
$U$ in it satisfies either $U^*\cap S^*\neq\emptyset$ or $U^*\subseteq
T^*$. By Theorem \ref{T.UniqueLam} 
the first alternative is equivalent to $U^*=S^*$ and the
second to $U^*=\emptyset$ or $U^*=T^*$. Since the set of trees $U$
with $U^*=S^*$ is closed and disjoint from the set where $U^*=T^*$,
the accumulation set must include some trees with $U^*=\emptyset$
\end{proof}

\subsection{Reducing Factors are Visible}

Recall the construction of folding lines from Section 2.5.

A train track structure on a graph is {\it recurrent} if there is a
legal loop crossing every edge, and it is {\it birecurrent} if there
is a legal loop that crosses every edge with either orientation.

\begin{lemma}\label{recurrent}
 Let $T \in \CVN$, and let $\Sigma$ be a fixed simplex in
 $\CVN$.  There exists a point
 $T_0$ in $\Sigma$ such that any optimal map $f:T_0 \to T$ induces a
 recurrent train track structure on $T_0/\FN$.
\end{lemma}

\begin{proof}
 Let $T_0 \in \Sigma$ be a point that minimizes
the function $\Sigma\to\R$, $Y\mapsto d(Y,T)$, by which we mean the
$\log$ of the smallest Lipschitz constant of a difference of markings
map $Y\to T$.
Such a point exists since the map is proper.
 
Let $G=T_0/\FN$ and consider an optimal map $f:T_0\to T$. First note
that the tension graph $\Delta$ must be all of $G$ for otherwise we
could increase the metric on $\Delta$ and decrease it in the
complement thus reducing the distance to $T$. Likewise, all vertices
of $\Delta$ must have $\geq 2$ gates, for otherwise we may perturb $f$
to another optimal map whose tenson graph is a proper subgraph, again
reaching contradiction.

Now consider the directed graph $D$ whose vertices are
oriented edges of $G$, where there is a directed edge from $e$ to $e'$ if
$G$ has a legal path of the form $e\cdots e'$. Two vertices in a
directed graph are equivalent if there are directed paths joining each
with the other. A finite directed graph always has an equivalence
class $S$ of vertices so that there are no directed edges from a
vertex in $S$ to a vertex outside $S$. Note that if $S$ contains some
edge $e$ with both orientations, then all edges in $S$ come with both
orientations. Now we have the following possibilities.

{\it Case 1.} $S=D$. Then the train track structure on $G$ is
birecurrent. 

{\it Case 2.} $S$ contains every edge of $G$ with a single
orientation. Then $G$ is recurrent (and has a coherent orientation). 

{\it Case 3.} $S$ consists of edges in a subgraph $G'$ with both
orientations. This is impossible since the 2 gate condition guarantees
that there is a legal path of length 2 with one edge in $G'$ and one
outside. 

{\it Case 4.} $S$ consists of edges in a subgraph $G'$ with a single
orientation. Orient the edges in $G'$ according to $S$ and note that
for every vertex of $G'$ all incoming edges must belong to the same
gate. Also orient all
half-edges outside $G'$ but incident to a vertex $v$ of $G'$ towards
$v$. All such half-edges must belong to the same gate as all incoming
edges within $G'$.

 We now perform a sort of backward flow relative to this orientation
 of $G'$.  Fix a small $\epsilon>0$. For any vertex $x \in G'$, flow
 $x$ backwards along an incoming edge for time $\epsilon$.  This gives
 a multi-valued function $\phi_\epsilon$ on the set of vertices. The
 ambiguity comes from the fact that there may be more than one
 incoming edge at $x$. In a similar way, we have a multivalued
 function $\tilde\phi_\epsilon$ defined on the set of vertices of
 $\tilde G'\subset T_0$ with values in $T_0$. However, all points in
 $\tilde \phi_\epsilon(x)$ are identified under folding, so the composition
 $f_\epsilon=f\tilde\phi_\epsilon$ is a well-defined function from the
 vertex set of $\tilde G'$ to $T$. On the vertices outside $\tilde G'$
 we define $f_\epsilon$ as $f$ and we then extend to the edges
 linearly. The slope on the edges in $\tilde G'$ or disjoint from
 $\tilde G'$ remains the same as before, and on the remaining edges
 the slope of $\tilde f_\epsilon$ is strictly smaller than
 before. Thus the tension graph is a proper subgraph of $G$ and we can
 change the metric as before to reduce the distance to $T$, a
 contradiction.
\end{proof}

\begin{remark}
The result continues to hold if $T$ is in $\partial\CVN$, provided it
does not belong to the closure of $\Sigma$. We do not need this
generalization. 
\end{remark}

\begin{lemma}\label{limit of folding paths}
Let $\Delta_0,\Gamma_i\in \CVN$ and assume $\Gamma_i$ converges to
$T\in\partial {CV}_N$. Let $\gamma_i$ be a folding path from
$\Delta_i$ to $\Gamma_i$, where $\Delta_i$ is in the same simplex as
$\Delta_0$ and is given by Lemma \ref{recurrent}, i.e. every optimal
map $\Delta_i\to\Gamma_i$ has a recurrent train track structure. Then
one of the following holds, after a subsequence.
\begin{enumerate}
\item [(i)] $\Delta_i$ converges to $\Delta\in \CVN$ and certain
  initial segments of $\gamma_i$ converge uniformly on compact sets to
  a folding path (ray) $\gamma$ from $\Delta$ that converges to $S\in
  \partial \CVN$ with $S^*\subseteq T^*$, or
\item [(ii)] $\Delta_i$ converges to a tree $S\in\partial\cvn$, and
  every element elliptic in $S$ is also elliptic in $T$.
\end{enumerate}
\end{lemma}

\begin{proof}
 Let $f_i:\Delta_i \rightarrow \Gamma_i$ be optimal maps.  After a
 subsequence, all $\Delta_i$ belong to the same open simplex and underlying
 graphs can be identified; only the metric depends on $i$.  If a
 subsequence of $\Delta_i$ projects to a sequence contained in a
 compact subset of $\CVN$, then we are in case (i).  Otherwise, the
 injectivity radii of $\Delta_i$ go to zero, so that after a
 subsequence $\Delta_i \rightarrow S$.  As $\Delta_i$ degenerate to
 $S$, there is a core subgraph $G \subseteq \Delta_i$ which is the
 union of loops whose lengths go to 0; its volume goes to 0.

 Pass to a subsequence so that all train track structures on
 $\Delta_i$ agree. Since they are recurrent, there is an element $g
 \in \FN$ whose representatives in $\Delta_i$ are legal and cross
 every edge of $\Delta_i$. Let $s$ be a loop contained in $G$, then
 $length_{\Gamma_i}(s)/length_{\Gamma_i}(g) \leq
 length_{\Delta_i}(s)/length_{\Delta_i}(g) \to 0$; hence elliptic
 elements in $S$ are also elliptic in $T$.

 Now suppose that we are in case (i), \emph{i.e.} after a subsequence,
 $\Delta_i$ converge to $\Delta \in \CVN$. If we show the convergence
 statement, then the claim $S^* \subseteq T^*$ will follow from
 Theorem \ref{U-turn3}. Parametrize all folding paths via arc
 length. Consider the set $$T=\{t_0\in [0,\infty)\mid
   \gamma_i|[0,t_0]\mbox{ converges uniformly after a subsequence}\}$$ It
   follows from the Arz\'ela-Ascoli Theorem, using the fact that small
   metric closed balls are compact, that small $t_0>0$ belong
   to $T$, and more generally, $T=[0,t_0)$ for some $t_0>0$ (possibly
     $t_0=\infty$). By a diagonal argument there is a subsequence so
     that $\gamma_i|[0,t_0)$ converges uniformly on compact sets to a
       ray $r_t$ in $\CVN$.  We show that $r_t$ is a folding path.
       The point here is that being a folding path is a local
       condition.

 For $t \in [0,t_0)$, we have that $\gamma_i(t)$ converge to $r_t$ so
   we have maps $r_t\to\gamma_i(t)$ and $\gamma_i(t)\to r_t$ that are
   $(1+\epsilon_i(t))$-Lipschitz with $\epsilon_i(t)\to 0$. Composing
   with these maps, we obtain for $0\leq t_1\leq t_2<t_0$ maps
   $f_{t_1,t_2}:r_{t_1}\to r_{t_2}$ as limits of folding maps
   $\gamma_i(t_1)\to\gamma_i(t_2)$ (this may require a further
   subsequence; e.g. do it for rational $t_1,t_2$ and arrange that for
   $t_1<t_2<t_3$ the map $f_{t_1,t_3}$ is the composition
   $f_{t_2,t_3}f_{t_1,t_2}$ and then define $f_{t_1,t_2}$ in general
   by taking limits). The limiting maps are optimal and there are at
   least two gates at every vertex, by a straightforward limiting
   argument.  

It remains to argue that $r_t$ is a (greedy) folding path 
when
restricted to $[0,t_2]$ for $t_2<t_0$, induced by the optimal map
$f_{0,t_2}$. It is convenient to rescale the graphs and reparametrize
$[0,t_2]$ so that all maps $f_{t_1,t_2}$ are isometric on small
segments. If edges $e_1,e_2$ in $r_{t_1}$ form an illegal turn, their
images in $r_{t_2}$ overlap on an initial segment of length
$>\epsilon$, say. For large $i$ the images $e_1',e_2'$ of $e_1,e_2$ in
$\gamma_i(t_1)$ are nearly isometric to $e_1,e_2$ and have possible
 overlap much less than $\epsilon$, but their images in
 $\gamma_i(t_2)$ have overlap $>\epsilon$. This means that for
 $0<\delta<\epsilon$ the images of $e_1',e_2'$ in $\gamma_i(t_1+\delta)$
 have overlap of about $\delta$ and by taking the limit we see that
 the turn $e_1,e_2$ is folding with speed 1.
\end{proof}

From Lemma \ref{limit of folding paths}, one has that if $T$ is
arational, and if $\Delta$ and $\Gamma_i$ are as
in the statement, then we always are in case (i).

The proof of the following three lemmas uses the Rips Theory; see
\cite{BF95} for background. It will be convenient to use the following
terminology.  A simple subgroup is
{\it reducing} for $T\in\partial\cvn$ if it acts with dense orbits on
a subtree of $T$.  A {\it quasi-surface} is a 2-complex $K$ obtained from a graph
$\Gamma$ by
attaching a collection of compact surfaces with negative Euler
characteristic along the boundary.

Let $K$ be a quasi-surface.  Note that when $\pi_1(K)$ is free, then $K$ has a ``collapsible boundary
component''. More precisely, fix an isomorphism $\pi_1(K)\cong F_N$,
represent $F_N$ by a rose $R$, and represent each component of the
underlying graph of $K$ by an immersion to $R$. Thus a map to $R$ is
also defined on the boundary of the attaching surfaces and we may
extend to each surface. After homotopy, the preimage of a regular
value $y$ will consist of trees. An endpoint indicates that $y$ is
crossed by only one boundary component, and only once. This is our
``collapsible'' boundary component -- for more details see \cite[Lemma
  4.1]{BF94}. Now
there are two possibilities. One is that this boundary component is
attached to a circle component of $\Gamma$ by a degree 1 map, which we
may take to be a homeomorphism. In this case the boundary component is
``free''. Otherwise, the boundary component is attached along arcs to
other parts of $K$ and cutting along these arcs produces a free
splitting of $F_N$ showing that all attached surfaces represent
simple subgroups. 

Using quasi-surfaces the reader can easily construct examples of trees
that satisfy any one of the three alternatives below, but not the
other two. 

\begin{lemma}\label{alternative}
Suppose $T$ is very small and has trivial arc stabilizers. Then either
\begin{enumerate}[(i)]
\item every point stabilizer is simple, or
\item there is a cyclic point stabilizer which is not simple and
  all other point stabilizers not conjugate to it are simple, or
\item $T$ has a reducing subgroup $A\in\mathcal R(T)$ such that $A|T$ is
  dual to a filling measured lamination on a compact surface with
  negative Euler characteristic.
\end{enumerate}
\end{lemma}

\begin{proof}
First assume that $T$ is geometric, i.e. dual to a measured lamination
on a finite 2-complex $K$. Then $K$ can be transformed using the Rips
machine into a standard form. If the lamination contains compact
leaves (which, in the standard form, are just points in an edge of
$K$), then (i) holds. Likewise, if the lamination contains a component
of thin (Levitt) type, one can find a morphism $T'\to T$ where $T'$
has the same set of elliptic elements and is dual to a lamination on a
finite complex that has compact leaves (by ``cutting a slit in a naked
band''), so again (i) holds. The remaining possibility is that all
components of the lamination are of surface type, i.e. $K$ is obtained
from a complex $K_0$ not carrying any leaves by attaching surfaces
carrying minimal laminations. Since each component of $K_0$ represents
a point stabilizer, which is free, we may replace it with a
graph. Thus $K$ is a quasi-surface.  By the above discussion, $K$
contains a collapsible boundary component. If this component is not
free, then (iii) holds. If it is free note that this free boundary
component and its powers are the only possible nonsimple elliptic
elements, since we may collapse from a free boundary component to a
complex of the form $K'$ wedge a homotopically nontrivial graph with
$K'$ carrying all other elliptic elements 
(since the surfaces have negative Euler characteristic).

When $T$ is not geometric, one can find a geometric resolution $T'\to
T$ so that $T'$ has the same set of elliptic elements. The tree $T'$
is also very small and has trivial arc stabilizers. If (i) or (ii)
holds for $T'$ then it also holds for $T$. If (iii) holds for $T'$
then $A|T'\to A|T$ is an isomorphism, so (iii) holds for $T$ as
well. The isomorphism statement is
Skora's theorem \cite{skora} that predates the Rips machine. For a
relatively simple proof using the Rips machine
see \cite{handbook}. The
idea is that any further folding of $A|T'$ would be resolved by a
complex obtained from $S$ by identifying
distinct leaves, and therefore would contain many leaves that contain
at least 4 disjoint rays. By the classification of measured
laminations this forces toral (axial) components, which are impossible
for very small trees.

Note that the Skora theorem also follows from \cite{CH10} as in
Proposition \ref{P.Maximal}.
\end{proof}

\begin{lemma}\label{arc stabs}
 If $G_t$ is a folding line that converges to $T \in \partial \cvn$
 and if there is a non-degenerate arc $I=[x,y] \subseteq T$ with
 non-trivial stabilizer, then either $Stab(I)=Stab(x)$ or
 $Stab(I)=Stab(y)$.
\end{lemma}

\begin{proof}
 Choose $a \in Stab(x)$, $b \in Stab(y)$, and let $c$ generate
 $Stab(I)$; so $ab$ and $bc$ are also elliptic in $T$.  We will show
 that $ac$ is also elliptic, which proves the claim.

 Note that if $g \in \FN$ is elliptic in $T$, then the length of
 $g|G_t$ is necessarily bounded; indeed, the number of illegal turns
 in $g|G_0$ is an upper bound for the number of illegal turns in
 $g|G_t$, so if $g|G_t$ is unbounded, then $g|G_t$ must contain a long
 legal segment.  Choose a basepoint $b \in G_{0}$; the images $b_t$ of
 $b$ in $G_t$ give basepoints in $G_t$.  We think of all elements of
 $\FN$ as loops based at $b_t$.  Choose graphs $H_a^t, H_b^t, H_c^t,
 H_{ab}^t, H_{bc}^t$ with immersions into $G_t$ representing
 $a,b,c,ab,bc$, respectively; each $H_i^t$ looks like a balloon,
 \emph{i.e.} a circle, possibly with a (long) segment, called a
 \emph{string}, attached to it.  After contracting the strings to a
 point, we get graphs of bounded size for all $t$.

 If all strings are short, then $bc$ is clearly represented by
 immersions of bounded size for all $t$, implying that $bc$ is
 elliptic in $T$.  If the strings for $H_a^t$ are not short, then
 since $ab|G_t$ is bounded, the string for $H_a^t$ contains all but a
 bounded amount of the string for $H_b^t$ for all $t$; similarly, the
 string for $H_b^t$ contains all but a bounded amount of the string
 for $H_c^t$.  Hence the string for $H_a^t$ contains all but a bounded
 amount of the string for $H_c^t$, and it follows that $ac$ is bounded
 in $G_t$ as well and is elliptic in $T$.
\end{proof}

\begin{lemma}\label{non-simple}
 The alternative in Lemma \ref{alternative} holds also for trees
 $T\in\partial \CVN$ that are limits of folding paths $G_t$.
\end{lemma}

\begin{proof}
In the geometric case the proof is the same except that now we allow
annuli with circle leaves as surfaces when building the
quasi-surface. If we find a free boundary component, then its surface
cannot be an annulus since otherwise the dual tree would not be
minimal, so (i) or (ii) holds as before. If we find a collapsible
boundary component, we can cut as before and 
(iii) holds. If all surfaces are annuli
$T$ is simplicial and Lemma \ref{arc stabs}
implies that $T$ has edges with trivial stabilizer, so (i) holds.  The nongeometric case follows as before.
\end{proof}

In the next lemma it is convenient to work with naturally parametrized
folding paths in $cv_N$.

\begin{lemma}\label{dense orbits}
Let $G_t$ be a folding path in $cv_N$ converging to a tree
$T\in\partial cv_N$. 
If $\lim vol(G_t)=0$, then $T$ has dense orbits.
\end{lemma}

\begin{proof}
We will use the fact that $T$ is the equivariant
Gromov-Hausdorff limit of $\tilde G_t$; see \cite{paulin}. The claim
would be essentially obvious if $G_t$ had a vertex with at least three
gates. In general we argue as follows.  For every $\epsilon>0$ there
is $t_0$ so that $vol(G_{t_0})<\epsilon$ and for every point $\tilde
x\in \tilde G_{t_0}$ there is a tripod $\{\tilde a,\tilde b,\tilde
c\}\subset \tilde G_{t_0}$ whose center is within $\epsilon$ of
$\tilde x$, $[\tilde a,\tilde b]$ and $[\tilde a,\tilde c]$ are legal,
and all three segments between the center and the endpoints have
length $>1$. For $t>t_0$ the images of $[\tilde a,\tilde b]$ and
$[\tilde a,\tilde c]$ may get folded past the image of the center, but
not by more than $vol(G_{t_0})$, see \cite{BFH}. Thus branch points
are dense in $T$. 
\end{proof}

\begin{lemma}\label{combined}
Suppose $G_t$, $t\in [0,\infty)$, is a folding path 
  converging to $T\in\partial \cvn$, and assume that $T$ is not
  arational.  Then there is a factor $B< F_N$ such that $B \in
  \mathcal{R}(T)$, and such that $B$ has uniformly bounded volume
  along $G_t$ for $t$ large.

In particular, the projection of $G_t$ to $\mathcal{F}$ is bounded and
for large $t_0$ the projection of $G_t$, $t\in [t_0,\infty)$ is
  uniformly bounded.
\end{lemma}

\begin{proof}
First we argue that it suffices to find a {\it simple} reducing subgroup
$A<F_N$ such that $vol(A|G_t)$ is uniformly bounded for large
$t$. Indeed, let $B$ be the smallest factor that contains $A$. Then
$A|G_t\to G_t$ factors through $B|G_t$ and $A|G_t\to B|G_t$ is
surjective (otherwise we find a smaller factor that contains $A$) so
$vol(B|G_t)\leq vol(A|G_t)$ is also bounded for large $t$.

Second, recall \cite[Corollary 3.5]{BF11} that if $B$ is a factor then
the distance in $\mathcal F$ between $B$ and $\pi(G_t)$ is bounded by
a function of $vol(B|G_t)$, so the last sentence follows from the
first paragraph.

Third, if the length of the folding path is finite, or equivalently if
after rescaling $G_t$ so that folding maps are isometries on small
segments the volume does not go to 0, we may take $A$ to be the
fundamental group of a
component of the subgraph whose volume goes to 0. So we will now
assume that the folding path has infinite length.

For now let $C$ be any simple subgroup reducing $T$.  It follows from
\cite[Lemma 4.1]{BF11} that $C|G_t$ cannot contain a legal segment of
length $>2$ inside a topological edge; otherwise the volume of $C|G_t$
would grow exponentially, and $C$ would not reduce $T$.

Fix a large number $M$, much larger than the possible number of
illegal turns in any train track structure. 
If the number of illegal turns in each topological edge 
in $C|G_t$ is $\leq M$ for large $t$, then $C|G_t$ has
uniformly bounded volume and we are done. 
Choose $t_0$ such that the number of
illegal turns in topological edges of $C|G_t$ has stabilized for $t
\geq t_0$ (by the Unfolding Principle of \cite{BF11} the number of
such turns cannot increase), and suppose that for some edges this
number is $>M$ and focus on $M$ consecutive such illegal turns. 
By our choice of $M$ there are many turns in this collection that
project to the same illegal turn in $G_t$.
This gives many loops
$g|G_{t}$ of uniformly bounded length for $t$ large.

For $t_n \to \infty$, we have scaling constants $\lambda_n$ such that
$\tilde G_{t_n}/\lambda_n$ converges to $T \in \partial \cvn$. From
our assumption that the folding path has infinite length we see that
$\lambda_n \to \infty$ and so any $g$ constructed above is elliptic in
$T$. 

If $C|G_{t}$ contains legal loops, consider the subgraph $D_t\subset
C|G_t$ which is the union of all legal loops. This subgraph is clearly
forward invariant and eventually the number of components and their
ranks stabilize. Take the simple subgroup $A$ to be represented by
one of these stable components. Then by Lemma \ref{dense orbits} $A$ is reducing, and $A|G_t$ has
uniformly bounded volume for all large $t$.

So we are done unless $C|G_t$ doesn't contain any legal loops for
large $t$, or equivalently
the complement in $C|G_t$ of the set of 1-gate vertices
is a forest. In this case $C$ is elliptic in $T$. Indeed, loops in
$C|G_t$ have uniformly bounded legal segments (sufficiently long legal
segments would close up to form legal loops) and so grow slower than
legal loops (see the Derivative Formula in \cite{BF11}).

We now consider three cases according to the three alternatives in
Lemma \ref{alternative} which also applies to $T$ by Lemma
\ref{non-simple}.  If (i) holds then the elements $g$ constructed
above are simple and we are done. If (iii) holds, we will start with
$C=A$, a non-elliptic reducing simple subgroup, and by the above
discussion we are done. Finally, assume (ii) holds. We constructed
paths $g_i$ in $C|G_t$ connecting consecutive equivalent illegal
turns. Under the assumption (ii), if $g_i$ are not simple, they are
all conjugate to powers of a fixed element $g$. Note that if
consecutive paths $g_i$, $g_{i+1}$ do not have common powers, then
$g_ig_{i+1}$ is not conjugate to a power of $g$, since $C$ is
simple. We conclude that the concatenation of the $g_i$'s is a large
power of an element conjugate to $g$, which we rename $g$. Now $C|G_t$
is not just a loop representing a power of $g$, so inside $C|G_t$ we
can find a loop representing an element of the form
$x_1g^{n_1}x_2g^{n_2}\cdots x_kg^{n_k}$ with $k$ and lengths of $x_i$
bounded, and this element is elliptic in $T$, not conjugate to a power
of $g$, and stays of bounded length along $G_t$. Here $n_i$ can be
large, but we may replace them with uniformly bounded numbers and the
new element is still elliptic in $T$, it is uniformly bounded in $G_t$
for large $t$, and it is simple since it is not conjugate to a power
of $g$.
\end{proof}

\section{The Boundary of the Complex of Free Factors}

Let $\partial \mathcal{F}$ denote the boundary of the complex of free
factors, and let $\mathcal{AT} \subseteq \partial \CVN$ denote the set
of (projective classes of) arational trees.  Define an equivalence relation
$\sim$ on $\mathcal{AT}$, where $T \sim S$ if and only if $L(T)=L(S)$
(equivalently, $T^*=S^*$, see Proposition \ref{P.ZeroInter} and
Theorem \ref{T.UniqueLam}).
We note that $\sim$ is precisely the relation of ``forgetting the
measure'' for elements of $\mathcal{AT}$; see \cite{CHL07}.  Give
$\mathcal{AT}$ the subspace topology, and consider the quotient map
$p:\mathcal{AT} \to \mathcal{AT}/\mysim$.

\begin{lemma}\label{L.ClosedMap}
 The quotient map $p:\mathcal{AT} \to \mathcal{AT}/\mysim$ is
 closed, and point pre-images are compact.
\end{lemma}

\begin{proof}
 Let $K \subseteq \mathcal{AT}$ be closed; we show that
 $C=p^{-1}(p(K))$ is closed.  Let $\{T_n\}$ be a convergent sequence
 in $C$, say $T_n$ converges to $T\in \mathcal{AT}$; let $Y_n \in K$
 such that $p(Y_n)=p(T_n)$.  This means that 
 $Y_n^*=T_n^*$. Now, let $\eta_n \in
 Y_n^*=T_n^*$. After passing to a further subsequence we may assume
 that $Y_n\to Y\in\partial\CVN$ and that $\eta_n\to\eta\in M_N$.  By
 Proposition \ref{P.Inter}, we have $\langle Y, \eta \rangle
 =0=\langle T,\eta \rangle$; so
 Theorem \ref{T.UniqueLam} gives that $Y$ is arational and that
 $L(T)=L(Y)$. It follows that $Y\in K$ and $p(T)=p(Y)$, so $T \in C$.

 The statement that equivalence classes are compact can be proved
 similarly using the compactness of $\partial\cvn$. If $T_i$ converge to $T$ in
 $\partial\cvn$ and if $T_i\in \mathcal{AT}$ are all equivalent, then
 choose some $\nu\in T_i^*$. By Proposition \ref{P.Inter} we have
 $\nu\in T^*$ and then $T\in \mathcal{AT}$ is equivalent to all $T_i$
 by Theorem \ref{T.UniqueLam}.
\end{proof}

The following result justifies our use of sequential arguments.

\begin{cor}
The quotient space $\mathcal{AT}/\mysim$ is metrizable and second countable.
\end{cor}

\begin{proof}
Closed surjective maps with compact point preimages preserve the
properties of being metrizable and second countable \cite[Theorems 3.7.19 and
  4.4.15]{engelking}.
\end{proof}

We can now give a description of $\partial \mathcal{F}$.

\begin{prop}\label{P.Defined}
 There is a continuous map $\partial \pi:\mathcal{AT} \to \partial
 \mathcal{F}$, such that if $T_i \in \CVN$ converge to $T \in
 \mathcal{AT}$, then $\pi(T_i)$ converge to $\partial \pi(T)$.
\end{prop}

\begin{proof}
 Let $T_i \in \CVN$ converge to $T \in \mathcal{AT}$; we need to see that $\pi(T_i)$
 converges to a point of $\partial \mathcal{F}$ that depends only on $T$.  Toward
 contradiction, suppose this is not the case. Then we get subsequences
 $X_n$ and $Y_n$ such that the Gromov product $(\pi(X_n),\pi(Y_n))$ is
 uniformly bounded.  Consider (say a standard) geodesic $[X_n,Y_n]$;
 Proposition \ref{P.Uniform} gives that these geodesics are mapped by
 $\pi$ to uniform quasi-geodesics in $\mathcal{F}$.  Hence we find
 $Z_n$ on $X_n \to Y_n$ with $\pi(Z_n)$ of uniformly bounded distance
 from any basepoint in $\mathcal{F}$.  On the other hand, Lemma
 \ref{defined and continuous} and Theorem \ref{T.UniqueLam} give that
 any limit $Z$ of $\{Z_n\}$ must be arational.  Finally, Corollary
 \ref{C.ArationalAtInfinity} gives a contradiction.  Hence, we have a function $\partial \pi: \mathcal{AT} \to \partial \mathcal{F}$.

 The continuity statement follows similarly. Let $T_i \in
 \mathcal{AT}$ converge to $T \in \mathcal{AT}$, but assume that
 $\partial\pi(T_i)$ does not converge to $\partial\pi(T)$. After a
 subsequence we may assume that $(\partial\pi(T_i),\partial\pi(T))$ is
 bounded above. Choose trees
 $X_i,Y_i\in \CVN$ so close to $T_i$ and $T$ respectively that
 $(\pi(X_i),\pi(Y_i))$ is also bounded above and so that $X_i\to T$,
 $Y_i\to T$. As above, there is $U_i$ on a geodesic from $X_i$ to
 $Y_i$ with $\pi(U_i)$ at bounded distance from a basepoint, which is impossible.
\end{proof}

\begin{prop}\label{P.Injective}
 For arational trees $S$ and $T$, we have $\partial \pi(S)=\partial
 \pi(T)$ if and only if $L(S)=L(T)$
\end{prop}

\begin{proof}
 By the same argument as in the proof of Proposition
 \ref{P.Defined}, we get that $L(S)=L(T)$ implies that $\partial
 \pi(S)=\partial \pi(T)$.  So assume that $L(S) \neq L(T)$, let $S_n$
 converge to $S$ and $T_n$ converge to $T$; consider standard
 geodesics $[S_n, T_n]$.  By Lemma \ref{injective}, we have that $[S_n, T_n]$ accumulates on some portion of $\CVN$, hence after passing
 to a subsequence, we find points on $[S_n, T_n]$ projecting to
 points of $\mathcal{F}$ of uniformly bounded distance from any base
 point.  Hence $(\pi(S_n),\pi(T_n))$ is uniformly bounded, so
 $\partial \pi(S) \neq \partial \pi(T)$.
\end{proof}

\begin{prop}\label{P.Surjective}
 The map $\partial \pi$ is surjective.  Further, if $\{T_n\}$ converge
 to a tree $T$ that is not arational, then no subsequence of
 $\{\pi(T_n)\}$ converges to a point of $\partial \mathcal{F}$.
\end{prop}

\begin{proof}
 Let $X \in\partial \mathcal{F}$, and let $X_n \in \mathcal{F}$
 converge to $X$.  Choose $T_n \in \pi^{-1}(X_n)$, and pass to a
 subsequence so that $\{T_n\}$ converges to $T$ in $\overline{CV}_N$.
 We will show that $T \in \mathcal{AT}$, which implies $\partial
 \pi(T)=X$.

 Toward contradiction, suppose that $T$ is not arational. Recall that
 $T$ has its reducing set $\mathcal R(T)\subset\mathcal F$, which is
 nonempty and uniformly bounded, see Corollary \ref{C.Bounded}.
 Fix $n$ large so that for $m>n$ $X_m$ belongs to a small neighborhood
 of the end $X$ and is far from $\mathcal R(T)$. In particular, we may
 assume that geodesics connecting $X_m$ and $X_{m'}$ for $m,m'>n$ are
 also far from $\mathcal R(T)$.

Now consider for $m>>n$ a folding path $[T'_m,T_m]$ where $T_m'$ is
in the same simplex as $T_n$ and the train track structure on
$T_m'/F_N$ is recurrent (see Lemma \ref{recurrent}),
and let $m \to
\infty$. Apply Lemma \ref{limit of folding paths}, and first assume
that case (ii) applies, so $T_m'\to S$ and elliptic elements in $S$
are elliptic in $T$. Then there is a factor $A$ which is elliptic
in $S$, and thus it is reducing for $T$, and it is also coarsely equal
to $\pi(T_n)$, contradicting the fact that $\pi(T_n)$ is far from
$\mathcal R(T)$.
  
 Now, suppose that case (i) of Lemma \ref{limit of folding paths}
 applies, so after a subsequence initial segments of $[T'_m, T_m]$
 converge to a ray $r_n$ that converges to $S$ with $S^* \subseteq
 T^*$.  By Corollary \ref{stick} $S$ is not arational and $\mathcal R(S)$
 coarsely equals $\mathcal R(T)$. 
 Using Lemma \ref{combined},
we see that the projection of $r_n$ to $\mathcal F$ is eventually
contained in a uniformly bounded neighborhood of $\mathcal R(S)$, and
therefore of $\mathcal R(T)$.

 To obtain a contradiction just note that the projections of $[T'_m, T_m]$ are uniform quasi-geodesics by Proposition \ref{P.Uniform}, so
 they don't come close to $\mathcal R(T)$ for large $n,m$ and the
 geodesics $T'_m\to T_m$ cannot accumulate to $r_n$.
\end{proof}

\begin{lemma}
The map $\partial\pi:\mathcal {AT}\to\partial\mathcal F$ is closed.
\end{lemma}

\begin{proof}
 Let $C \subseteq \mathcal{AT}$ be closed, and let $K=\partial
 \pi(C)$.  Let $X_n \in K$ converge to $X \in \partial \mathcal{F}$;
 we want to find $Y \in C$ with $\partial \pi(Y)=X$.  Choose $Y_n \in
 (\partial \pi)^{-1}(X_n) \cap C$, and pass to a subsequence to ensure
 that $Y_n$ converge to $Y \in \partial \CVN$.

 We claim that $Y\in \mathcal{AT}$. This follows from
 Proposition \ref{P.Surjective} applied to a sequence $\{T_n\}$ in
 $\CVN$ approximating $\{Y_n\}$ so that $T_n\to Y$ and $\pi(T_n)\to
 X$. Now the fact that $\partial \pi(Y)=X$ follows from Proposition \ref{P.Defined}.
\end{proof}

Summarizing, we have:

\begin{thm}\label{T.Main}
 The space $\partial \mathcal{F}$ is homeomorphic to the quotient
 space $\mathcal{AT}/\mysim$.
\end{thm} 

\begin{proof}
 The map $\partial \pi:\mathcal{AT} \to \partial \mathcal{F}$ factors through $p:\mathcal{AT} \to \mathcal{AT}/\mysim$ to give a continuous, bijective, closed map $\mathcal{AT}/\mysim \to \partial \mathcal{F}$.
\end{proof}

\bibliography{indecompREF}

\end{document}